\PassOptionsToPackage{reqno}{amsmath}
\documentclass[11pt,oneside]{amsart}

\usepackage[T1]{fontenc}
\usepackage[utf8]{inputenc}
\usepackage{graphicx}
\usepackage{graphics}
\usepackage{amssymb}
\usepackage{amstext}
\usepackage{amsthm}
\usepackage{amsfonts}
\usepackage{amsmath}

\usepackage{xcolor}

\usepackage{hyperref}
\hypersetup{
    colorlinks,
    linkcolor={red!50!black},
    citecolor={blue!50!black},
    urlcolor={blue!80!black}
}

\theoremstyle{plain}
\newtheorem{teorema}{Theorem}
\newtheorem{lemma}[teorema]{Lemma}

\newtheorem{proposizione}[teorema]{Proposition}
\newtheorem{corollary}[teorema]{Corollary}
\newtheorem{quest}{Question}

\theoremstyle{definition}

\theoremstyle{remark}
\newtheorem{remark}[teorema]{Remark}
\newtheorem{esempio}[teorema]{Example}

\newcommand{\noshow}[1]{}

\newcommand{\norm}[1]{\lvert #1 \rvert}

\title{Comparing word probabilities \\ in finite groups}

\author[L.\ Putignano]{Lorenzo Putignano}
\address{L.\ Putignano: Okinawa Institute of Science and Technology, Okinawa, Japan.
\emph{email: lorenzo.putignano@oist.jp}}

\author[M.\ Vannacci]{Matteo Vannacci}
\address{M.\ Vannacci: Dipartimento di Matematica `Ulisse Dini', Universit\`a degli Studi di Firenze, Italy. \emph{email: matteo.vannacci@unifi.it}}

\thanks{The second named author is funded by the Italian program Rita Levi Montalcini for young researchers (Edition 2021)}

\begin{document}

\begin{abstract}
    We use results from extremal graph theory and analysis to compare word probabilities in finite groups. In particular, we show that the probability of satisfying the 2-Engel word admits a positive lower bound in terms of the probability that an element has order dividing 3. More generally, using a group removal lemma, we obtain analogous bounds for $x^m$-good words.
\end{abstract}

\maketitle

\section{Introduction}
 The purpose of this article is to use techniques from extremal graph theory and analysis in the study of probabilistic group theory.  
 
 Let $G$ be a finite group, and let $F_d$ be the free group with basis $\{x_1,\ldots,x_d\}$.
 Every word $w=w(x_1,\ldots,x_d)$ in $F_d$ induces the \emph{word map} $w:G^d\longrightarrow G$ given by $(g_1,\dots,g_d)\mapsto w(g_1,\ldots,g_d)$. We denote by $P(G,w)$ the probability that $w$ evaluates to the identity element of $G$ on a uniformly chosen $d$-tuple; that is,
\[
  P(G,w)=\frac{\lvert\{(g_1,\ldots,g_d)\in G^d\mid w(g_1,\ldots,g_d)=1_G\}\rvert}{|G|^d}.
\] 

A classical result of Gustafson \cite{gust} shows that if $P(G,[x,y])> 5/8$, then the group $G$ is abelian. Since Gustafson's result, many similar statements have been proved for other words; see, for instance, \cite{lp,malekan1,malekan2,Sh}. Motivated by these, we might expect that if $w\in F_d$ and a sufficiently large proportion of $d$-tuples $(g_1,\ldots,g_d)$ of $G^d$ satisfy $w(g_1,\ldots,g_d)=1$, then the group $G$ is `close' to satisfying the identity $w$. This viewpoint is closely related to the theory of probabilistic identities; see Section~\ref{sec:LS}.
    
In this article, we take a slightly different point of view. Rather than studying a single word probability, we compare the probabilities associated with different word maps. Our main question is the following. 
\begin{quest}\label{quest}
 Let $w_1$ and $w_2$ be  two words. Is there a non-decreasing function $f:(0,1]\to (0,1]$ such that the inequality $P(G,w_2)\ge f(P(G,w_1))$ holds  for every finite group $G$? 
\end{quest}
In this article, we give positive answers to the previous question for several pairs of words $(w_1,w_2)$. Our arguments combine combinatorial and analytic methods and yield quantitative bounds.

In Section~\ref{sec2}, we compare the word probabilities associated with $w_1(x)=x^3$ and the \emph{2-Engel word} $w_2(x,y) = [x,y,y]$. In Section~\ref{sec:triangle}, we use a triangle removal lemma (see Theorem~\ref{thm:trirem}) to prove that $P(G,[x,y,y])$ admits a positive lower bound in terms of $P(G,x^3)$. In Section~\ref{sec:elem_prob}, we give an elementary probabilistic argument that yields an explicit bound when $P(G,x^3)>1/2$.

In Section~\ref{general}, we generalise Proposition~\ref{prop:cubetoengel} using a group removal lemma (see Theorem~\ref{thm:removal}). More precisely, we show that Question~\ref{quest} has a positive answer for every pair $(x^m,w)$ such that $w$ is a \emph{$x^m$-good word}, in the sense introduced in Section~\ref{sec:good}. In Section~\ref{sec:constants}, we discuss the quantitative bounds arising from this argument and show that, when $m$ is even, the group removal lemma can be avoided and one obtains an explicit polynomial bound for $P(G,w)$.



Finally, in Section~\ref{sec:commutators} we prove that Question~\ref{quest} has a positive answer for the pairs $([x,y],[u,v])$ where $[x,y]$ is the commutator word and $u$ and $v$ are two words in disjoint variables. Here, more than the result itself, it is interesting to note that the proof uses Jensen's inequality, which to the best of our knowledge has never been used to study probability in finite groups. 



\subsection{Probabilistic identities}\label{sec:LS}
The motivation for Question~\ref{quest} comes from the theory of probabilistic identities in profinite groups. Any profinite group can be equipped with its normalised Haar measure $\mu$, making it a probability space. If $G$ is residually finite with profinite completion $\widehat G$, we extend the notation $P(G,w)$ by setting 
\[P(G,w)=\mu(\{(g_1,\ldots,g_k)\in \widehat{G}^k \mid w(g_1,\ldots,g_k)=1\}).\]
A word $w\in F_k$ is said to be a probabilistic identity in a residually finite group $G$ if $P(G,w)>0$.

The main conjecture about probabilistic identities is the Larsen--Shalev Conjecture \cite[Problem~10]{LS}: if a finitely generated residually finite group $G$ satisfies a probabilistic identity, then $G$ satisfies an identity. 

Let $w_1$ and $w_2$ be two words, and suppose that $w_1$ is a probabilistic identity in a finitely generated residually finite group $G$. Write $p=P(G,w_1)>0$. Let $(N_i)_{i\geq 1}$ be a descending sequence of finite-index normal subgroups of $G$ with trivial intersection. Then, by \cite[p.\ 211, (11.2)]{LuSe}, $P(G,w_1)=\inf_i P(G/N_i,w_1)$ and so $P(G/N_i,w_1)\geq p$
for every $i$. Hence, if $f$ is as in Question~\ref{quest}, then
\[
P(G/N_i,w_2)\geq f(P(G/N_i,w_1))\geq f(p),
\]
where the second inequality follows from the fact that $f$ is non-decreasing. Again, by \cite[p.\ 211, (11.2)]{LuSe}, we obtain
 that $P(G,w_2)\geq f(p)>0$. Thus, $w_2$ is also a probabilistic identity in $G$. This explains the monotonicity assumption in Question~\ref{quest}.


Hence, if the Larsen–Shalev conjecture is known for probabilistic identities defined by $w_2$, then the same conclusion follows for $w_1$. Thus, a positive answer to Question~\ref{quest} may allow one to transfer known cases of the Larsen--Shalev conjecture from $w_2$ to $w_1$. By Proposition~\ref{prop:cubetoengel}, this transfer applies to the words $w_1=x^3$ and $w_2=[x,y,y]$. However, to the best of our knowledge, it is not known whether the Larsen--Shalev conjecture holds for $[x,y,y]$.


\subsection{Notation and preliminaries}

We collect here the notation and preliminaries used throughout the article. 

In this article, we reserve the notation $P(G,w)$ for word probabilities, while $\mathbb P$ denotes the probability measure on a general probability space $(\Omega, \mathbb{P})$. 

We regard every finite group as a probability space equipped with the uniform probability measure: for $g\in G$, $\mathbb{P}(\{g\})=1/\lvert G\rvert$. The Cartesian power $G^d$ is equipped with the corresponding product probability measure.

A word $w$ in $d$ variables is an \emph{identity} in a group $G$ if $w(g_1,\ldots,g_d)=1$ for all $g_1,\ldots,g_d \in G$. Given a word $w$, we denote by \[w(G)=\langle w(g_1,\ldots,g_d)\mid g_1,\ldots,g_d\in G\rangle\] the \emph{verbal subgroup} of $w$ in $G$. We observe that $w$ is an identity in $G$ if and only if $w(G)=1$.

Given two elements $x$ and $y$ of a group $G$, the conjugate of $x$ by $y$ is denoted by $x^y$ and defined by $x^y=y^{-1}xy$. The \emph{commutator} of $x$ and $y$ is $[x,y]=x^{-1}y^{-1} xy = (y^x)^{-1} y$. The \emph{$2$-Engel} word is the nested commutator $[x,y,y]=[[x,y],y]$.

\subsection*{AI declaration and history of the results}

The present work started while the first author visited the second author in 2021 at the University of the Basque Country. In August 2026 a first draft of the article was prepared. This first version was basically the idea in Section~\ref{sec:elem_prob} below and it contained a gap, as outlined below, found by the authors. 

In a conversation with Gemini (version 3.1 Pro), the authors learned about the existence of the triangle removal lemma. The AI system did not produce a proof. Subsequently, the authors worked independently and developed the results contained in the article. 

Hence, the AI use was limited to: 1) the suggestion of a possibly interesting connection between word probabilities and extremal graph theory, 2) proofreading and language corrections in the late stages of the preparation of this article. 

\section{Comparing $x^3$ and 2-Engel probabilities}\label{sec2}
In this section, we compare the word probabilities associated with $x^3$ and the 2-Engel word. In Section ~\ref{sec:triangle}, we use the triangle removal lemma to obtain a positive lower bound for $P(G,[x,y,y])$ in terms of $P(G,x^3)$. In Section~\ref{sec:elem_prob}, we give an elementary probabilistic argument which yields an explicit quantitative bound when $P(G,x^3)>1/2$. 

\subsection{A bound via the triangle removal lemma}\label{sec:triangle}

We prove the following result.

\begin{proposizione}\label{prop:cubetoengel}
  There exists a non-decreasing function $f:(0,1] \to \mathbb (0,1]$ such that, for every finite group $G$,
  \[
    P(G,[x,y,y]) \ge f(P(G,x^3)).
  \]
\end{proposizione}

The proof uses the following classical result from extremal graph theory, known as the \emph{triangle removal lemma}. It was first proved in \cite{RS}; the formulation used here can be found in \cite{CF}. We recall that a \emph{triangle} in a graph is a closed circuit of length $3$ without backtracking and a graph is \emph{triangle-free} if it does not contain any triangle. 

\begin{teorema}\label{thm:trirem}
For every $\varepsilon> 0$, there exists $\delta=\delta_{\triangle}(\varepsilon)> 0$ such that any graph on $n$ vertices with at most $\delta n^3$ triangles can be made triangle-free by removing at most $\varepsilon n^2$ edges.
\end{teorema}

\begin{remark}\label{rmk:nondec}
 In what follows, we will take $\delta_\triangle(\varepsilon)$ to be the \emph{optimal} triangle-removal constant; i.e.,\ $\delta_\triangle(\varepsilon)$ is the supremum of all constants for which Theorem~\ref{thm:trirem} holds for a fixed $\varepsilon$. With this convention, it is clear that $\delta_\triangle(\varepsilon) \le \delta_\triangle(\eta)$ for $\varepsilon\le \eta$; i.e.\ $\delta_\triangle$ is a non-decreasing function.
\end{remark}

We also need the following well-known result; see \cite[Thm.~12.3.5]{robinson}. We include a short proof for convenience.

\begin{lemma}\label{lem:local}
If $a^3=(ab)^3=(a^{2}b)^3=1,$ then $[a,b,b]=1$.
\end{lemma}
\begin{proof}
The collection formula gives
\[
  (ab)^3=b^{a^2}b^a b=1, \quad \text{and}\quad  (a^2b)^3=b^{a}b^{a^2} b=1.
\]
Hence $b^{a^2}$ and $b^a$ commute. Conjugating by $a^2$, we see that $b$ commutes with $b^a$ and therefore with $(b^{a})^{-1}$. Since $[a,b]=(b^a)^{-1}b$, it follows that $[a,b]$ commutes with $b$, proving the lemma.
\end{proof}

\begin{proof}[Proof of Proposition \ref{prop:cubetoengel}]
Let $S=\{g\in G\mid g^3=1\}$ and $\varepsilon=P(G,x^3)$. By definition, $\lvert S\rvert = \varepsilon  \lvert G\rvert$.
We observe that $S=S^{-1}$ and hence we 
can form the undirected Cayley graph $\Gamma=\mathrm{Cay}(G,S\smallsetminus\{1\})$. Note that there are $\lvert G\rvert(\lvert S \rvert-1)/2$ unordered edges in $\Gamma$. Define a \emph{basic triangle} in $\Gamma$ to be of the form $g\langle s\rangle=\{g,gs,gs^2\}$ for $g\in G$ and $s\in S\smallsetminus\{1\}$. We want to count the number of edge-disjoint basic triangles.

We will show that each edge in $\Gamma$ belongs to a unique basic triangle. First, we observe that given an edge $\{g,gs\}$ in $\Gamma$, the subgroup $\langle s\rangle$ is uniquely determined by this edge as $\langle s \rangle=\langle g^{-1}gs\rangle = \langle (gs)^{-1}g \rangle$. Then, if $\{g,gs\}\subseteq h\langle t\rangle$ for $g,h\in G$ and $s,t\in S\smallsetminus\{1\}$, then we must have that $\langle t \rangle=\langle s \rangle$. So we can conclude that the basic triangles $h\langle t\rangle$ and $g\langle s\rangle$ coincide. Thus the basic triangles partition the edge set of $\Gamma$. Since each basic triangle has three edges, their number is
\begin{equation}\label{eq:numb_edges}
  \frac{\lvert G\rvert(\lvert S\rvert-1)}{6}.
\end{equation}

Suppose first that $\lvert G\rvert>2/\varepsilon $. Then $
\lvert S\rvert -1= \varepsilon \lvert G\rvert-1> \varepsilon \lvert G\rvert/2$. So by Equation~\eqref{eq:numb_edges} more than $\varepsilon \lvert G\rvert^2/12$ edges must be deleted to make $\Gamma$ triangle-free. Consequently, writing $\tau(\Gamma)$ for the number of unordered triangles in the Cayley graph $\Gamma$, by the contrapositive of Theorem \ref{thm:trirem} with $\varepsilon'=\varepsilon /12$, we must have
\[
\tau(\Gamma)> \delta_\triangle \left( \frac{\varepsilon }{12}\right)\lvert G\rvert^3.
\]

A rooted and oriented triangle can be written uniquely as $(g,gp,gpq)$ with $ p,q,pq\in S\smallsetminus\{1\}$. Therefore, since every unordered triangle has six rooted orientations, we have
\[
  \tau(\Gamma) =\frac{1}{6} \lvert G\rvert \cdot \lvert \{(p,q)\in G^2\mid p,q,pq\in S\smallsetminus\{1\}\} \rvert,
\]
and there are at least $6\delta_\triangle(\varepsilon /12)\lvert G\rvert^2$ such pairs $(p,q)$. For each such pair, put $a=p$, $b=p^{-1}q$.
Since the cubes of $a=p$, $ab=q$ and $a^2b=pq$ are all equal to 1, Lemma~\ref{lem:local} applies and gives $[a,b,b]=1$. Moreover, since the map $(p,q)\mapsto(p,p^{-1}q)$ is bijective, it follows that
\begin{equation}\label{eq:firstineq}
  P(G,[x,y,y])\ge 6\delta_\triangle\left(\frac{\varepsilon}{12}\right).
\end{equation}

If instead $\lvert G\rvert\le2/\varepsilon $, the $\lvert G\rvert$ pairs $\left\{(g,1)\mid g\in G\right\}$ already give 
\begin{equation}\label{eq:secondineq}
P(G,[x,y,y])\ge \frac{1}{\lvert G\rvert}\ge\frac{\varepsilon }{2}.\end{equation}

Therefore, by \eqref{eq:firstineq} and \eqref{eq:secondineq}, we conclude that $P(G,[x,y,y])\ge f(\varepsilon )$, where

\[
f(t)=\min\left\{\frac{t}{2},\,6\delta_{\triangle}\left(\frac{t}{12}\right)\right\}
\qquad \text{for } t\in(0,1].
\]
Now, by Remark~\ref{rmk:nondec}, both functions occurring in the minimum are non-decreasing, and hence so is $f$.
\end{proof}


\subsection{Independence and word fibres}\label{sec:elem_prob}

One might wonder whether Proposition~\ref{prop:cubetoengel} can be proved directly using independence. The three natural events arising from the argument are pairwise independent but need not be mutually independent, so the naive product estimate does not apply. Nevertheless, pairwise independence yields an elementary lower bound for $P(G,[x,y,y])$ when $P(G,x^3)>1/2$.

Recall that events $A_1,\ldots,A_n$ in a probability space $(\Omega,\mathbb{P})$ are said to  be \emph{mutually independent} if $\mathbb{P}(\bigcap_{j=1}^{s} A_{i_j})=\prod_{j=1}^s \mathbb{P}(A_{i_j})$ for all $1\le i_1<\ldots<i_s\le n$ with $1\le s\le n$. 

\begin{lemma}\label{lem:indep}
    Let $S\subseteq G$. Then the events $A=S\times G$, $B=\{ (x,y)\in G^2 \mid xy \in S \}$ and $C=G\times S$ are pairwise independent.
\end{lemma}
\begin{proof}
    The independence of $A$ and $C$ is immediate. We show that $A$ and $B$ are independent; the proof for $B$ and $C$ is analogous.
    
    First, since left-multiplication by an element of $G$ is a bijection of $G$, it is clear that $ \lvert B \rvert = \lvert G \rvert \cdot \lvert S\rvert$.
    Moreover, if a pair $(x,y)$ lies in the intersection $A \cap B$, then $x\in S$ and $xy\in S$. Hence, $x\in S$ and $y\in x^{-1}S$, so $\lvert A\cap B \rvert = \lvert S\rvert^2$. 
    Then $\mathbb{P}(A\cap B) = \lvert S \rvert^2/\lvert G \rvert^2 = \mathbb{P}(A)\cdot \mathbb{P}(B)$ and the result follows.
\end{proof}

\begin{remark}\label{rmk:mutally}
Let $S=\{g\in G\mid g^3=1\}$, and let $A$, $B$, $C$ be as in Lemma~\ref{lem:indep}. Then $A$, $B$, $C$ are pairwise independent and \[\mathbb{P}(A)=\mathbb{P}(B)=\mathbb{P}(C)=P(G,x^3).\]
If, in addition, the events $A$, $B$, $C$ were mutually independent, we would have that 
\begin{align*}
  P(G,[x,y,y]) & \ge \mathbb{P}(A\cap B\cap C) = \mathbb{P}(A) \mathbb{P}(B) \mathbb{P}(C)
  = P(G,x^3)^3,
\end{align*}
where the inequality follows by the bijection
$(x,y)\to (x,x^{-1}y)$ together with Lemma~\ref{lem:local}.

However, the events $A$, $B$, $C$ need not be mutually independent in general, as the following example shows. 
\end{remark}
\begin{esempio}
Consider the alternating group $A_4$. There are exactly $9$ elements of order dividing $3$, so $P(A_4,x^3)=3/4$. On the other hand, the elements $a$ of $A_4$ such that $a^3\neq 1$ are exactly the non-trivial elements of the Klein $4$-subgroup. By direct calculation, we obtain
\[
   \lvert \{(x,y)\in {A_4}^2 \mid x^3=y^3=(xy)^3=1\} \rvert = 57.
\]
Hence, $\mathbb{P}(A\cap B\cap C) = 57/144< \left(3/4\right)^3$.
\end{esempio}

We now use pairwise independence to obtain a direct lower bound for $\mathbb{P}(A\cap B\cap C)$ for groups where more than half of the elements have order dividing 3. 

\begin{lemma}\label{lem:pofint}
Let $A$, $B$ and $C$ be pairwise independent events in a probability space  $(\Omega, \mathbb{P})$ with $p=\mathbb{P}(A)=\mathbb{P}(B)=\mathbb{P}(C)$. Then
    $$ \mathbb{P}(A\cap B\cap C) \ge  \max\{0,2p^2-p\}.$$
\end{lemma}
\begin{proof}
    By hypothesis, the events $A$, $B$ and $C$ are pairwise independent. Hence, $\mathbb{P}(A\cap B)=\mathbb{P}(B\cap C)= \mathbb{P}(A\cap C)=p^2$. Moreover, writing $q= \mathbb{P}(A\cap B \cap C)$, we have
    \begin{align*}
       0\le \mathbb{P}(A \cap B^c \cap C^c) & = \mathbb{P}(A) - \mathbb{P}(A \cap B) - \mathbb{P}(A \cap C) + \mathbb{P}(A \cap B \cap C) 
       \\ & = p-2p^2 +q.
    \end{align*}
    Then $q\ge 2p^2-p$, and since $q\ge0$, the result follows.
\end{proof}

\begin{corollary}\label{cor:half}
Let $\varepsilon\in (1/2,1]$ and $G$ be a finite group. If $P(G,x^3)=\varepsilon$, then $P(G,[x,y,y])\ge 2\varepsilon^2-\varepsilon$.
\begin{proof}
In the notation of Remark~\ref{rmk:mutally}, by the argument in the proof of Proposition \ref{prop:cubetoengel} and Lemma \ref{lem:pofint} we have that $P(G,[x,y,y])\ge \mathbb{P}(A\cap B\cap C)\ge 2\varepsilon^2-\varepsilon$.
\end{proof}
\end{corollary}


Unlike the bound obtained in Section~\ref{sec:triangle}, the bound in Corollary~\ref{cor:half} is explicit and follows from an elementary probabilistic argument. It applies, however, only when $P(G,x^3)>1/2$. This condition is rather restrictive; in particular, such groups are solvable; see \cite[Theorem~1.2(1)]{MY}.

\section{A generalisation via the group removal lemma}\label{general}
It is natural to ask whether an implication between identities induces a corresponding relation between word probabilities. More precisely, suppose that $w_1(G)=1$ implies $w_2(G)=1$ for two words $w_1,w_2$, for every finite group $G$. Does a uniform positive lower bound for $P(G,w_1)$ necessarily yield a uniform positive lower bound for $P(G,w_2)$? The following example shows that this is false in general, even when $w_1$ and $w_2$ define the same \emph{variety of groups} (in the sense of \cite{NeumannNeumann}).

\begin{esempio}\label{ex:Q2n}
    Consider the words $w_1(x,y)=x^2 y^{-2}$ and $w_2(x)=x^2$. Then, for any group $G$, $w_1(G)=1$ if and only if $w_2(G)=1$; that is, $w_1$ and $w_2$ define the same variety of groups. On the other hand, for $n\ge3$, consider the generalised quaternion group 
    \[
        Q_{2^n} = \left \langle a,b \mid a^{2^{n-1}}=1,\ b^2=a^{2^{n-2}},\ b^{-1}ab= a^{-1} \right \rangle. 
    \]
    Every element outside $\langle a\rangle$ has square $b^2$. A direct count gives $P(Q_{2^n},w_1)=1/4+3/2^n\ge 1/4$. Moreover, $b^2$ is the only element of order $2$, so
    \[
      P(Q_{2^n},w_2) = \frac{2}{2^n} \to 0 \quad \text{as } n\to \infty.
    \]
    \end{esempio}

This example shows that even if \(w_1\) and \(w_2\) define the same variety of groups, a fixed positive lower bound for $P(G,w_1)$ does not produce a uniform positive lower bound for $P(G,w_2)$. Thus, an implication between the corresponding identities is not sufficient to control the associated word probabilities. 

In the rest of this section we will prove a generalisation of Proposition~\ref{prop:cubetoengel} to other words.

\subsection{\texorpdfstring{$x^m$}{xm}-good words}\label{sec:good}

We now introduce a stronger condition, tailored to the group removal lemma (Theorem \ref{thm:removal} below), which keeps track of how one identity follows from a system of power relations.

Fix $m\ge3$, and consider the free group $F_{m-1}=\langle z_1,\ldots,z_{m-1}\rangle$. Let $R_m$ be the normal closure of $z_1^m,\ldots,z_{m-1}^m, (z_1\cdots z_{m-1})^m$ in $F_{m-1}$. We say that a word $w\in F_{m-1}$ is \emph{$x^m$-good} if there exists an automorphism $\alpha \in \mathrm{Aut}(F_{m-1})$ such that $\alpha(w)\in R_m$. Note that if $w$ is $x^m$-good, then every tuple $(g_1,\ldots,g_{m-1})$ of elements of a group $G$ satisfying $g_1^m = \ldots = g_{m-1}^m = (g_1\cdots g_{m-1})^m=1$ also satisfies $\alpha(w)(g_1,\ldots,g_{m-1})=1$.
Since $\alpha$ is an automorphism of $F_{m-1}$, \cite[Fact~1.1]{HMP} gives, for every finite group $G$, $P(G,w)=P(G,\alpha(w))$. 

\begin{esempio}[$x^3$-good words]\label{ex:x3good}
For $m=3$, we have $R_3=\langle z_1,z_2,(z_1 z_2)^3 \rangle^{F_2}$. Consider the automorphism $\alpha\in \mathrm{Aut}(F_2)$ defined by $\alpha(z_1)=z_1$ and $\alpha(z_2)=z_1^{-1}z_2$.
Modulo $R_3$, Lemma~\ref{lem:local} applied with $a=z_1$ and
$b=z_1^{-1}z_2$ shows that $b$, $b^{z_1}$ and $b^{z_1^2}$
commute pairwise. Since these elements generate the image of
$\alpha(\langle z_2\rangle^{F_2})$ modulo $R_3$, this image is abelian.
Hence, $\alpha([\langle z_2 \rangle^{F_2},\langle z_2 \rangle^{F_2}]) \le R_3$. It follows that for any two words $u,v\in \langle z_2 \rangle^{F_2}$, the commutator $[u,v]$ is $x^3$-good.

Since both $[z_1,z_2]=(z_2^{-1})^{z_1} z_2$ and $z_2$ are in $\langle z_2 \rangle^{F_2}$, it follows that $[z_1,z_2,z_2] = [[z_1,z_2],z_2]$ is $x^3$-good. 
Writing $e_1(z_1,z_2)=[z_1,z_2]$ and $e_{n+1}(z_1,z_2)=[e_n(z_1,z_2),z_2]$, it follows inductively that $e_n(z_1,z_2)$ is $x^3$-good for all $n\ge2$.

In the same way, one can show that any element of $F_2''$ is $x^3$-good.
\end{esempio}

The main result of this section is the following.

\begin{proposizione}\label{prop:good}
 Let $m\ge3$. Then there exists a non-decreasing function $\rho_m:(0,1] \to \mathbb (0,1]$ such that, for every $x^m$-good word $w\in F_{m-1}$ and every finite group $G$, 
 \[
   P(G,w) \ge \rho_m(P(G,x^m)).
 \]
\end{proposizione}

The proof relies on the following `group removal lemma'.

\begin{teorema}[{\cite[Theorem~2]{KSV}}]\label{thm:removal}
  Let $G$ be a finite group. Let $A_1,\ldots, A_m$, $m \ge 2$, be subsets of $G$ and let $g\in G$. If the equation $x_1 \cdots x_m = g$ has $o(\lvert G\rvert^{m-1})$ solutions with $x_i \in A_i$, then there are subsets $A_i' \subseteq A_i$ with $\lvert A_i \smallsetminus A_i'\rvert = o(\lvert G\rvert)$ such that there is no solution of the equation $x_1 \cdots x_m = g$ with $x_i \in A_i'$.
\end{teorema}

We shall use the following quantitative formulation of Theorem~\ref{thm:removal}.  For every $m\ge2$ and every $\eta>0$, there exists a constant $\gamma_m(\eta)>0$ such that, for every finite group $G$ and every $A_1,\ldots,A_m\subseteq G$, if
\[
   \lvert \{(a_1,\ldots,a_m)\in A_1\times \ldots \times A_m \mid a_1 \cdots a_m =1\}\rvert <\gamma_m(\eta)\lvert G\rvert^{m-1},
\]
then one can remove at most $\eta|G|$ elements from each $A_i$ so that no product-one tuple remains. Furthermore, $\gamma_m(\eta)$ does \emph{not} depend on $G$, but only on $m$ and $\eta$.

\begin{remark}\label{rmk:nondec2}
    As in Remark~\ref{rmk:nondec}, one can easily see that considering the optimal group-removal constant, the function $\gamma_m$ is non-decreasing.
\end{remark}

\begin{proof}[Proof of Proposition \ref{prop:good}]
Let $S=\{s\in G\mid s^m=1\}$ and $\varepsilon=P(G,x^m)$. By definition, $\lvert S\rvert = \varepsilon\lvert G\rvert$. We apply the quantitative formulation of Theorem~\ref{thm:removal} above with $A_i=S$ for all $i$ and $\eta=\varepsilon/(2m)$.

Suppose that there are fewer than $\gamma_m(\eta)\lvert G\rvert^{m-1}$ product-one tuples in $S^m$. Then we can delete sets $D_i \subseteq S$ with $\lvert D_i\rvert \le \eta\lvert G\rvert$ and eliminate every solution. Since
\[
  \lvert D_1 \cup \ldots \cup D_m \rvert \le \sum_{i=1}^m \lvert D_i\rvert \le m\eta\lvert G\rvert=\frac{\varepsilon\lvert G\rvert}{2}<\lvert S\rvert,
\]
there exists $s\in S$ such that $s\in S\smallsetminus D_i$ for every $i$. It follows that the tuple $(s,\ldots,s)\in (S\smallsetminus D_1)\times\dots\times (S\smallsetminus D_m)$ is a solution of the equation $x_1 \cdots x_m = 1$, a contradiction. 

Since $S=S^{-1}$, we have that
\[  \lvert \{(s_1,\ldots,s_{m})\in S^{m} \mid \prod_{i=1}^m s_i =1 \}\rvert=
    \lvert \{(s_1,\ldots,s_{m-1})\in S^{m-1} \mid \prod_{i=1}^{m-1} s_i \in S\}\rvert. \]
Define \[\rho_m(t)=\gamma_m\left(\frac{t}{2m}\right)
\qquad \text{for } t\in(0,1].\]
Since $\eta=\varepsilon/(2m)$, the above cardinality is at least $\rho_m(\varepsilon)\lvert G\rvert^{m-1}$. Moreover, every such tuple $(s_1,\ldots,s_{m-1})$ satisfies the defining relations of $R_m$, so $\alpha(w)(s_1,\ldots,s_{m-1})=1$, as we noted earlier. Thus, $P(G,w)=P(G,\alpha(w))\ge \rho_m(\varepsilon)$. By Remark~\ref{rmk:nondec2}, $\rho_m$ is non-decreasing, and the result follows. \end{proof}

We conclude by observing that, taking into account Example~\ref{ex:x3good}, Proposition \ref{prop:good} gives an alternative proof of Proposition \ref{prop:cubetoengel}. 

\subsection{Quantitative bounds}\label{sec:constants}

The group-removal argument above yields a very weak quantitative bound. Indeed, if $\gamma_m(\eta)$ denotes the removal constant, known estimates give
\[
   \gamma_m(\eta)\ge \frac{1}{\mathrm{twr}_2(\lceil C_m \log(2/\eta)\rceil) },
\]
where $C_m>0$ is a computable constant, $\mathrm{twr}_2(n+1)=2^{\mathrm{twr}_2(n)}$ and $\mathrm{twr}_2(0)=1$ (see \cite{fox}).

However, in the special case where $m$ is even, the group removal lemma can be completely avoided, and one can take the polynomial function $\rho_m(t)=t^m.$

Let $m=2k$, set $S=\{s\in G\mid s^m=1\}$ and $\varepsilon=P(G,x^m)$, and define, for $g\in G$, \[r(g)= \lvert\{(s_1,\ldots,s_k)\in S^k \mid s_1\cdots s_k=g \}\rvert.\]
Since $S^{-1}=S$, the map $(s_1,\ldots,s_k)\to (s_k^{-1},\ldots,s_1^{-1})$ is a bijection of $S^k$, hence $s_1\cdots s_k=g$ if and only if $s_k^{-1}\cdots s_1^{-1}=g^{-1}$. Thus $r(g^{-1})=r(g)$. Consequently,
\[
   T=\lvert\{(s_1,\ldots,s_m)\in S^m \mid s_1\cdots s_m=1 \}\rvert = \sum_{g\in G} r(g)r(g^{-1}) = \sum_{g\in G} r(g)^2. 
\]
By the Cauchy--Schwarz inequality \cite[Thm.~3.5]{rudin}, we have 
\[\Bigg(\sum_{g\in G} r(g)\Bigg)^2\le\sum_{g\in G} r(g)^2\cdot \sum_{g\in G} 1^2.\] Since the rightmost sum equals $\lvert G\rvert$, it follows that
\[
  T \ge \frac{1}{\lvert G\rvert} \Bigg(\sum_{g\in G} r(g) \Bigg)^2 =\frac{1}{\lvert G\rvert}(\lvert S\rvert^k)^2= \frac{\lvert S\rvert^{m}}{\lvert G\rvert} = \varepsilon^m \lvert G \rvert^{m-1},
\] where the last equality follows from the definition $\lvert S\rvert = \varepsilon \lvert G\rvert$. Arguing as in the proof of Proposition~\ref{prop:good}, it follows that $P(G,w)\geq\varepsilon^m$. Therefore, when $m$ is even, one may take $\rho_m(t)=t^m$.

\section{Commutators of word values}\label{sec:commutators}

The aim of this section is to show that a lower bound on the commuting probability of $G$ also gives lower bounds for the probability that the values of two arbitrary words commute. This might seem trivial, but as Example~\ref{ex:Q2n} shows, surprising things can happen when working with word maps. Below we use \emph{Jensen's inequality} to study the probability that two arbitrary word values commute in a finite group and we prove the following.

\begin{proposizione}

Let $r,s\geq 1$, let $u(x_1,\ldots,x_r)\in F_r$ and $v(y_1,\ldots,y_s)\in F_s$ be words in disjoint sets of variables, and let $G$ be a finite
group. Then
\[
P\bigl(G,[u(x_1,\ldots,x_r),v(y_1,\ldots,y_s)]\bigr)
\geq P(G,[x,y])^{rs}.
\]  
\end{proposizione}
\begin{proof}
Let $x_1,\ldots,x_r,y_1,\ldots,y_s$ be independent random elements of $G$, and put
$\mathbf{x}=(x_1,\ldots,x_r)$.
For $\mathbf{x}\in G^r$, define
\[c(\mathbf{x})=\frac{1}{|G|}\left|\bigcap_{i=1}^r C_G(x_i)\right|.\]
Thus $c(\boldsymbol{x})$ coincides with $P([x_i,y]=1\text{ for all }i \mid \boldsymbol{x})$, that is, the conditional probability that a random element $y$ of $G$ commutes with all the coordinates of a given $\boldsymbol{x}$. Since $y_1,\ldots,y_s$ are independent,
\[P\bigl([x_i,y_j]=1 \text{ for all }i,j\mid \mathbf{x} \bigr)=c(\mathbf{x})^s.\]
Averaging over $\mathbf{x}\in G^r$ and applying Jensen's inequality \cite[Thm.~3.3]{rudin} to the convex function $t\to t^s$, $t\ge0$, we
obtain
\[P\bigl([x_i,y_j]=1\text{ for all }i,j\bigr)=\frac{1}{|G|^r}\sum_{\mathbf{x}\in G^r} c(\mathbf{x})^s\geq\left(\frac{1}{|G|^r}\sum_{\mathbf{x}\in G^r} c(\mathbf{x})\right)^s.\]

For $y\in G$, put
\[d(y)=\frac{|C_G(y)|}{|G|}.\]
Double counting pairs $(\mathbf{x},y)\in G^r\times G$ such that $y$
commutes with every coordinate of $\mathbf{x}$ gives
\[\frac{1}{|G|^r}\sum_{\mathbf{x}\in G^r} c(\mathbf{x})=\frac{1}{|G|}\sum_{y\in G} d(y)^r.\]
Applying Jensen's inequality to the convex function $t\to t^r$, we obtain
\[\frac{1}{|G|}\sum_{y\in G} d(y)^r\geq\left(\frac{1}{|G|}\sum_{y\in G} d(y)\right)^r.\]
Since
\[\frac{1}{|G|}\sum_{y\in G} d(y)=\frac{1}{|G|^2}\sum_{y\in G}|C_G(y)|=P(G,[x,y]),\]
it follows that
\[P\bigl([x_i,y_j]=1\text{ for all }i,j)\geq P(G,[x,y])^{rs}.\]

Finally, write $X$ and $Y$ for the subgroups of $G$ generated by $\{x_1,\dots,x_r\}$ and $\{y_1,\dots,y_s\}$, respectively. If each $x_i$ commutes with each $y_j$, then $[X,Y]=1$, and $u(x_1,\ldots,x_r)\in X$ commutes with $v(y_1,\ldots,y_s)\in Y$. Consequently,
\[P\bigl(G,[u(x_1,\ldots,x_r),v(y_1,\ldots,y_s)]\bigr)\geq P(G,[x,y])^{rs},\]
as required.
\end{proof}


\begin{thebibliography}{9}

\bibitem{malekan2}
A.\ Abdollahi, M.\ Malekan.
\newblock Profinite groups with many elements of bounded order.
\newblock {\em Adv.\ Group Theory Appl.}, 13 (2022), 71--81.


\bibitem{CF}
D.\ Conlon, J.\ Fox.
\newblock Graph removal lemmas.
\newblock In {\em Surveys in Combinatorics 2013}, vol.\ 409 of London Mathematical Society Lecture Note Series (2013), 1--50.

\bibitem{fox}
J.\ Fox.
\newblock A new proof of the graph removal lemma.
\newblock {Ann.\ Math.}, 174 (2011), 561--579.

\bibitem{gust}
W.H.\ Gustafson.
\newblock What is the Probability that Two Group Elements Commute?
\newblock {\em Amer.\ Math.\ Monthly}, 80:9 (1973), 1031--1034.

\bibitem{HMP}
L.~Hanany, C.~Meiri, D.~Puder.
\newblock Some orbits of free words that are determined by measures on finite groups. 
\newblock {\em J.\ Algebra}, 555 (2020), 305--324.

\bibitem{KSV}
D.\ Kr\'al, O.\ Serra, L.\ Vena.
\newblock A combinatorial proof of the Removal Lemma for Groups.
\newblock {J.\ Comb.\ Theory Ser.\ A}, 116:4 (2009), 971--978.

\bibitem{lp}
L.\ L{\'e}vai, L.\ Pyber.
\newblock Profinite groups with many commuting pairs or involutions.
\newblock {\em Arch.\ Math.}, 75:1 (2000), 1--7.

\bibitem{LS}
M.\ Larsen, A.\ Shalev.
\newblock A probabilistic Tits alternative and probabilistic identities.
\newblock {\em Algebra Number Theory}, 10:6 (2016), 1359--1371.

\bibitem{LuSe}
A.\ Lubotzky, D.\ Segal.
\newblock Subgroup Growth, volume 212 of {\em Progress in Mathematics} (2003), Birkh\"auser, Basel.

\bibitem{malekan1}
M.\ Malekan, A.\ Abdollahi and M.\ Ebrahimi.
\newblock Compact groups with many elements
of bounded order.
\newblock {\em J.\ Group Theory}, 23 (2020), 991--998.

\bibitem{MY}
R.\ McCulloch, L.\ T.\ Young.
\newblock Finite groups with many elements of the same order.
\newblock Preprint (arxiv.org/abs/2602.19340), 2026.

\bibitem{NeumannNeumann}
H.\ Neumann.
\newblock Varieties of Groups. First edition (1967), Springer-Verlag, Berlin.


\bibitem{robinson}
D.J.S.\ Robinson.
\newblock A Course in the Theory of Groups.
\newblock Graduate Texts in Mathematics, vol.\ 80 (1996), Springer-Verlag.


\bibitem{rudin}
W.\ Rudin.
\newblock Real and Complex Analysis.
\newblock 3rd edition, McGraw-Hill (1987).

\bibitem{RS} 
I.Z.~Ruzsa and E.~Szemer\'edi.
\newblock Triple systems with no six points
carrying three triangles.
\newblock Combinatorics. Coll.\ Math.\ Soc.\ J.\ Bolyai 18, Volume II (1976), 939–945. 

\bibitem{Sh}
A.\ Shalev. 
\newblock Probabilistically nilpotent groups.
\newblock {\em Proc.\ Amer.\ Math.\ Soc.} 146 (2018), 1529--1536.
\end{thebibliography}
\end{document}